\documentclass[reqno]{amsart}
\usepackage[numbers,square,comma,sort&compress]{natbib}
\usepackage{amsmath,amssymb}
\usepackage{color}

\newcommand{\real}{\mathbb{R}}

\newcommand{\rn}{\real^N}

\newcommand{\eps}{\varepsilon}
\newcommand{\ep}{\epsilon}
\newcommand{\de}{\delta}

\newcommand{\D}{\Delta}
\newcommand{\lam}{\lambda}

\newcommand{\wt}{\widetilde}
\newcommand{\diff}{\,\mathrm{d}}

\renewcommand\emptyset{\mbox{\Large \o}}
\newcommand{\sm}{\setminus}
\newcommand{\x}{\times}
\newcommand{\les}{\leqslant}
\newcommand{\ges}{\geqslant}

\newcommand{\calI}{\mathcal{I}}

\newcommand{\calS}{\mathcal{S}}
\newcommand{\calC}{\mathcal{C}}

\newcommand{\disp}{\displaystyle}

\DeclareMathOperator \di{div}

\let\le=\leqslant
\let\ge= \geqslant

\newtheorem{theorem}{Theorem}
\newtheorem{lemma}{Lemma}
\newtheorem{proposition}[theorem]{Proposition}
\newtheorem{corollary}[theorem]{Corollary}
\newtheorem{remark}{Remark}

\begin{document}

\title[Bifurcation for the $p$-Laplacian]
{Bifurcation along curves for the $p$-Laplacian with radial symmetry}
\author{Fran\c cois Genoud}
\thanks{This work was supported by the Engineering and Physical Sciences 
Research Council [EP/H030514/1].} 
\address{Department of Mathematics and the Maxwell Institute for
Mathematical Sciences, Heriot-Watt University,
Edinburgh EH14 4AS, United Kingdom.}
\email{F.Genoud@ma.hw.ac.uk}

\begin{abstract}
We study the global structure of the set of radial solutions of the nonlinear 
Dirichlet problem
\begin{equation}\label{eq}\tag{1}
\setlength\arraycolsep{0.05cm}
\left\{ \begin{array}{rlll}
-\D_p (u) &= &\lam f(|x|,u) & \  \text{in} \ \Omega,\\
u &= &0 & \ \text{on} \ \partial \Omega,
\end{array}\right.
\end{equation}
where $\D_p(u):=\di(|\nabla u|^{p-2}\nabla u)$ is the $p$-Laplacian,
$p>2$, and $\Omega$ is the unit ball in $\rn$, $N\ge1$. 
The function $f$ is continuous, differentiable in its second argument, 
and satisfies $f(r,0)=0$ for all $r\in[0,1]$.
We show that all non-trivial radial solutions of \eqref{eq} lie
on smooth curves of respectively positive and negative solutions, parametrized by 
$\lam>0$, and bifurcating from the line of trivial solutions. This involves a local
bifurcation result of Crandall-Rabinowitz type, and global continuation arguments 
relying on monotonicity properties of $f$.
Furthermore, by prescribing the behaviour of $f(r,\xi)$ as $|\xi|\to\infty$, we control
the asymptotic behaviour of large solutions.
An important part of the analysis is dedicated to the delicate issue of differentiability of 
the inverse $p$-Laplacian.

We thus obtain a complete description of the
global continua of positive/negative solutions bifurcating from the first eigenvalue
of a weighted, radial, $p$-Laplacian problem, 
by using purely analytical arguments, whereas previous related
results were proved by topological arguments or a mixture of analytical and topological 
arguments. Our approach
requires stronger hypotheses but yields much stronger results, bifurcation occuring
along smooth curves of solutions, and not only connected sets.

\end{abstract}

\maketitle

\section{Introduction}

In this paper we consider the nonlinear Dirichlet problem
\begin{equation}\label{dir.eq}
\setlength\arraycolsep{0.05cm}
\left\{ \begin{array}{rlll}
-\D_p (u) &= &\lam f(|x|,u) & \  \text{in} \ \Omega,\\
u &= &0 & \ \text{on} \ \partial \Omega,
\end{array}\right.
\end{equation}
where $\D_p(u):=\di(|\nabla u|^{p-2}\nabla u)$ is the $p$-Laplacian, 
$p>1$, $\lam \ge0$, and $\Omega$ is the unit ball in $\rn$, $N\ge1$. 
The function $f$ is continuous, such that $f(r,0)=0$ for all $r\in[0,1]$, and will be subject
to various additional assumptions.

We look for $C^1$ radial solutions by studying the
problem
\begin{equation}\label{dirad.eq}
\setlength\arraycolsep{0.05cm}
\left\{ \begin{array}{rlll}
-(r^{N-1}\phi_p(u'))' &= &\lam r^{N-1} f(r,u), \quad 0<r<1,\\
u'(0) = u(1) &= &0,
\end{array}\right.
\end{equation}
where $\phi_p(\xi):=|\xi|^{p-2}\xi, \ \xi\in\real$, and $'$ denotes 
differentiation with respect to $r$. By a solution of \eqref{dirad.eq}
will be meant a couple $(\lam,u)$,
with $\lam \in \real$ and $u\in C^1[0,1]$, such that $\phi_p(u')\in C^1[0,1]$, that satisfies \eqref{dirad.eq}.
Note that, since $f(r,0)=0$ for all $r\in[0,1]$, $(\lam,0)$ is a solution for all $\lam\in\real$. Such solutions will be called trivial. We are interested in existence and bifurcation of 
non-trivial solutions of \eqref{dirad.eq}.

Bifurcation results for quasilinear equations in bounded domains have been considered
for instance in \cite{gs,dm,t} --- further references can be found in these papers. Del Pino
and Man\'asevich \cite{dm} prove global bifurcation from the first eigenvalue of the 
$p$-Laplacian in a general bounded domain, and global bifurcation from every 
eigenvalue in the radial case. They also obtain nodal properties of solutions along the 
bifurcating continua. These results generalize the well known results of 
Rabinowitz \cite{rab} to the quasilinear setting, using degree theoretic arguments. 

More recently, Girg and Tak\'a\v c \cite{t}
obtained results in the spirit of Dancer \cite{dan}, about bifurcation from the
first eigenvalue of an homogeneous quasilinear problem, 
in the cones of positive and negative solutions. 
They consider a large class of quasilinear problems in a general bounded domain 
$\Omega$ and they allow the
asymptotic problems as $|u|\to0/\infty$ to depend on $x\in\Omega$. 
They also prove their
results using topological arguments, combined with a technical asymptotic analysis.

The last contribution we want to mention here, which is probably the most closely related to our work, is the paper by Garc\'ia-Meli\'an and Sabina de Lis \cite{gs}. The famous
Crandall-Rabinowitz theorem \cite{cr} is extended in \cite{gs} to $p$-Laplacian
equations, in the radial setting, see \cite[Theorem~1]{gs}. 
This result yields a continuous local branch of solutions
bifurcating from every eigenvalue of the $p$-Laplacian, and uniqueness of the branch
in a neighbourhood of the bifurcation point. Garc\'ia-Meli\'an and 
Sabina de Lis then use this local result to obtain further information about the global 
structure of the continua of solutions obtained by the topological method in \cite{dm}.
In particular, they show that there exist
two unbounded continua $\calC^\pm$ of respectively
positive and negative solutions, which only meet at the bifurcation point 
$(\lam_0,0) \in \real\x C^1[0,1]$,
where $\lam_0$ is the first eigenvalue of the homogeneous problem
\begin{equation}\label{hom}
\setlength\arraycolsep{0.05cm}
\left\{ \begin{array}{rlll}
-(r^{N-1}\phi_p(u'))' &= &\lam r^{N-1} \phi_p(u), \quad 0<r<1,\\
u'(0) = u(1) &= &0.
\end{array}\right.
\end{equation}
Since \cite[Theorem~1]{gs} is only a local result, it is not known from \cite{gs} whether
the global continua $\calC^\pm$ are actually continous curves or only connected sets.
In fact, the picture obtained from \cite{gs} is somewhat hybrid, due to a mixture of
analytical arguments (essentially the implicit function theorem) used to get local 
bifurcation, and the topological method yielding the global continua $\calC^\pm$
in \cite{dm}.

Our main purpose in this paper is to show that, under additional assumptions
on the function $f$ in \eqref{dirad.eq}
--- in particular monotonicity assumptions ---, 
it is possible to describe the global structure of
solutions bifurcating from the first eigenvalue using purely analytical arguments. In fact,
we obtain smooth curves of respectively positive and negative 
solutions, parametrized by the bifurcation parameter $\lam$. 

Besides, we consider 
a more general homogeneous problem than \eqref{hom} in the limit $|u|\to0$. In fact, 
we allow both asymptotics as $|u|\to0/\infty$ to depend on $r\in[0,1]$, in the same spirit
as \cite{t}. The asymptotic problems as $|u|\to0/\infty$ --- see
equations \eqref{eigen} below --- are weighted homogeneous problems 
respectively associated with the asymptotes 
$$
f_0(r):=\disp\lim_{\xi\to0}\frac{f(r,\xi)}{\phi_p(\xi)}>0 \quad\text{and}\quad
f_\infty(r):=\lim_{|\xi|\to\infty}\frac{f(r,\xi)}{\phi_p(\xi)}>0, \quad r\in[0,1].
$$ 
The properties of ($\mathrm{E}_0$) enable us to obtain a 
local bifurcation theorem as in \cite{gs}, while
the asymptotic problem ($\mathrm{E}_\infty$) governs
the behaviour as $|u|\to\infty$. 

We will consider two different situations. In the first case, we will assume that $f(r,\xi)>0$
for all $(r,\xi)\in[0,1]\x\real^*$ and $f(r,0)=0$ for all $r\in[0,1]$. It follows that the set
of non-trivial solutions of \eqref{dirad.eq} is a smooth curve of positive solutions ---
see Theorem~\ref{main.thm}. If we rather assume that $f(r,\xi)\xi>0$ for 
$(r,\xi)\in[0,1]\x\real^*$ and $f(r,0)\equiv0$, 
then we get two smooth curves of respectively positive and
negative solutions, containing all non-trivial solutions of \eqref{dirad.eq} --- see
Theorem~\ref{main2.thm}.
Furthermore, if $N=1$, we
are also able to deal with the case where $f$ is `sublinear' at infinity, that is,
$f(r,\xi)/\phi_p(\xi)\to0$ as $|\xi|\to\infty$, uniformly for $r\in[0,1]$.

A one-dimensional problem similar to \eqref{dirad.eq} was studied by Rynne in 
\cite{r10}, from which the present work is substantially inspired. In particular, the explicit
form we get for the inverse $p$-Laplacian in the radial setting allows us to study
the differentiability of this operator following arguments of \cite{br}, where the
one-dimensional $p$-Laplacian was considered. This differentiability issue is probably
the most delicate part of the analysis. It should be noted that the results regarding the
inverse $p$-Laplacian in Section~\ref{inverse.sec} hold for any $p>1$, while we had
to restrict ourselves to $p>2$ in the bifurcation analysis for other differentiability
reasons --- see Remark~\ref{diff.rem}.

We conclude this section by a brief description of the content of the paper. In 
Section~\ref{results.sec}, we give some information about the
functional setting, our precise hypotheses, 
and we state our main results, Theorems~\ref{main.thm} and
\ref{main2.thm}. Then, in Section~\ref{inverse.sec}, we study an integral operator
corresponding to the inverse of the $p$-Laplacian in \eqref{dirad.eq}. The main 
results about this operator are Theorems~\ref{Spcont.thm} and \ref{Sp.thm}. It should
be noted that \cite{gs} already dealt with differentiability results similar to those of
Theorem~\ref{Sp.thm}. However, we believe that the discussion in \cite{gs} is incomplete
and so Theorem~\ref{Sp.thm} is of importance in its own right. In
Section~\ref{prop.sec}, we establish some {\it a priori} properties of solutions
of \eqref{dirad.eq}, notably positivity/negativity, as well as the asymptotic  behaviour of
solutions  $(\lam,u)$ as $|u|\to0/\infty$. 
Section~\ref{cranrab.sec} is devoted to the local bifurcation analysis,
where we establish, in particular, a Crandall-Rabinowitz-type result, 
Lemma~\ref{cranrab.lem}. Finally, the proofs of Theorems~\ref{main.thm} and
\ref{main2.thm} are completed in Section~\ref{global.sec}, where we show that
the local branches of solutions obtained in Section~\ref{cranrab.sec} can be extended
globally.


\section{Setting and main results}\label{results.sec}

We will work in various function spaces. We will denote by $L^1(0,1)$
the Banach space of real Lebesgue integrable functions
over $(0,1)$ and by $W^{1,1}(0,1)$ the Sobolev space of functions $u\in L^1(0,1)$ 
having a weak derivative $u' \in L^1(0,1)$.
$C^n[0,1]$ will denote the space of $n$ times continuously differentiable 
functions, with the usual sup-type norm $|\cdot|_n$.

In our operator formulation of \eqref{dirad.eq}, 
it will be convenient to use the shorthand notation
$$
X_p := \{u\in C^1[0,1] : \phi_p(u') \in C^1[0,1] \ \text{and} \ u'(0)=u(1)=0\} ,
\quad Y:=C^0[0,1].
$$

An important part of our discussion in the next section will concern the
differentiability of an integral operator, that will depend on the value of $p>1$.
This analysis will rely on results in \cite{br}, and we borrow the following notation
from there:
\begin{equation}\label{bp.eq}
B_p := \begin{cases} C^1[0,1], \quad 1< p \le 2,\\
W^{1,1}(0,1), \quad p>2.
\end{cases}
\end{equation}

However, our main results require $p\ge2$ and, apart from Section~\ref{inverse.sec},
we will suppose $p>2$ throughout the paper --- the results are well known for $p=2$.

Denoting by $\partial_2f$ the partial derivative of $f$ with respect to $\xi\in\real$,
we make the following hypotheses on the continuous function $f:[0,\infty)\x\real\to\real$:
\begin{itemize}
\item[(H1)] $f(r,\cdot)\in C^1(\real)$ for all $r\in [0,1]$ and 
$\partial_2f \in C^0([0,1]\x\real)$;
\item[(H2)] $f(r,\xi)>0$ for $(r,\xi) \in [0,1]\x\real^*$ and $f(r,0)\equiv0$;
\item[(H3)] $(p-1)f(r,\xi) \ges \partial_2f(r,\xi)\xi$ for $(r,\xi) \in [0,1]\x[0,\infty)$, and there
exist $\de,\ep>0$ such that $(p-1)f(r,\xi) > \partial_2f(r,\xi)\xi$ 
for all $(r,\xi)\in(1-\de,1]\x (0,\ep)$.
\end{itemize}
It follows from (H3) that, for any fixed $r\in[0,1]$, the mapping 
$\xi \mapsto f(r,\xi)/\phi_p(\xi)$ 
is decreasing on $(0,\infty)$. Therefore, there exist functions $f_0, f_\infty:[0,1]\to\real$ 
such that
$$f(r,\xi)/\phi_p(\xi) \to f_{0/\infty}(r) \quad\text{as} \ \xi \to 0^+/\infty$$
and
\begin{equation}\label{gbounds}
0\les f_\infty(r) \les f(r,\xi)/\phi_p(\xi) \les f_0(r) 
\quad\text{for all} \ (r,\xi) \in [0,1]\x(0,\infty).
\end{equation}
We will further assume that $f_0, f_\infty \in C^0[0,1]$ and (for $p>2$)
\begin{itemize}
\item[(H4)] $\disp\lim_{\xi\to0^+}|f(\cdot,\xi)/\phi_p(\xi) - f_0|_0 
= \lim_{\xi\to0^+}|\partial_2 f(\cdot,\xi)/\xi^{p-2} - (p-1) f_0|_0 = 0$;
\item[(H5)] $\disp\lim_{\xi\to\infty}|f(\cdot,\xi)/\phi_p(\xi) - f_\infty|_0 = 0$.
\end{itemize}

\begin{remark}\label{sign.rem}\rm
\item[(i)] Note that (H2) and \eqref{gbounds} imply $f_0>0$ on $[0,1]$.
\item[(ii)] Also, (H3) implies that, for $r\in(1-\de,1]$, 
the function $\xi \mapsto f(r,\xi)/\phi_p(\xi)$ 
is not constant on $(0,\infty)$. In particular, $f_\infty \not\equiv f_0$ on $[0,1]$.
\end{remark}

To state our main results, 
we need to relate problem \eqref{dirad.eq} to the homogeneous 
eigenvalue problems corresponding to the asymptotes $f_0, f_\infty \in C^0[0,1]$:
\begin{equation}\label{eigen}\tag{$\mathrm{E}_{0/\infty}$}
\setlength\arraycolsep{0.05cm}
\left\{ \begin{array}{rlll}
-(r^{N-1}\phi_p(v'))' &= &\lam r^{N-1} f_{0/\infty}(r)\phi_p(v), \quad 0<r<1,\\
v'(0) = v(1) &= &0.
\end{array}\right.
\end{equation}

The following result follows from \cite[Sec.~5]{w}.
\begin{lemma}\label{eigen.lem}
If $f_{0/\infty}>0$ on $[0,1]$ then problem \eqref{eigen} has a simple 
eigenvalue $\lam_{0/\infty}>0$ with a corresponding eigenfunction
$v_{0/\infty}>0$ in $[0,1)$, and no other eigenvalue having a positive 
eigenfunction. Furthermore, $f_\infty \ \substack{\les \\ \not\equiv} \ f_0$ implies 
$\lam_0<\lam_\infty$.
\end{lemma}

We know from Remark \ref{sign.rem} that $f_0>0$ and $f_\infty \ \substack{\les \\ \not\equiv} \ f_0$. For $\lam_\infty$ to be well-defined, we will still make the following assumption.
\begin{itemize}
\item[(H6)] Either\\ 
(a) $N\ge1$ is arbitrary and $f_\infty>0$ on $[0,1]$, or\\
(b) $N=1$ and $f_\infty\equiv0$ on $[0,1]$.\\
If (a) holds, $\lam_\infty>0$ is defined in Lemma~\ref{eigen.lem}; 
if (b) holds, we set $\lam_\infty=\infty$.
\end{itemize}

We are now in a position to state our first result about the solutions of \eqref{dirad.eq}.
From now on, we will refer to the collection of hypotheses (H1) to (H6) as (H).

\begin{theorem}\label{main.thm}
Suppose that $p > 2$. If (H) holds,
there exists 
$u\in C^1((\lam_0,\lam_\infty), Y)$ such that $u(\lam) \in X_p$, $u(\lam)>0$ on $[0,1)$ 
and, for any $\lam \in (\lam_0,\lam_\infty)$, $(\lam,u(\lam))$ is the unique 
non-trivial solution of \eqref{dirad.eq}. Furthermore,
\begin{equation}\label{asympt.eq}
\lim_{\lam\to\lam_0} |u(\lam)|_0=0 \quad\text{and}\quad 
\lim_{\lam\to\lam_\infty} |u(\lam)|_0=\infty.
\end{equation}
\end{theorem}

We will see in the proof of Theorem~\ref{main.thm} that the condition (H2) forces
the solutions of \eqref{dirad.eq} to be positive. If, instead of (H2) to (H5), 
we suppose:
\begin{itemize}
\item[(H2')] $f(r,\xi)\xi>0$ for $(r,\xi) \in [0,1]\x\real^*$ and $f(r,0)\equiv0$;
\item[(H3')] in addition to (H3),
$(p-1)f(r,\xi) \le \partial_2f(r,\xi)\xi$ for $(r,\xi) \in [0,1]\x(-\infty,0]$ and
$(p-1)f(r,\xi) < \partial_2f(r,\xi)\xi$ for all $(r,\xi)\in(1-\de,1]\x (-\ep,0)$;
\item[(H4')] $\disp\lim_{\xi\to0}|f(\cdot,\xi)/\phi_p(\xi) - f_0|_0 
= \lim_{\xi\to0}|\partial_2 f(\cdot,\xi)/|\xi|^{p-2} - (p-1) f_0|_0 = 0$;
\item[(H5')] $\disp\lim_{|\xi|\to\infty}|f(\cdot,\xi)/\phi_p(\xi) - f_\infty|_0 = 0$,
\end{itemize}
then the solutions need not be positive any more and we have the following result.
We refer to the collection of hypotheses (H1), (H2') to (H5') and (H6) as (H').

\begin{theorem}\label{main2.thm}
Suppose that $p > 2$. If (H') holds,
there exist $u_\pm\in C^1((\lam_0,\lam_\infty), Y)$ such that 
$u_\pm(\lam) \in X_p$, $\pm u_\pm(\lam)>0$ on $[0,1)$ and, 
for any $\lam \in (\lam_0,\lam_\infty)$, $(\lam,u_\pm(\lam))$ are the only
non-trivial solutions of \eqref{dirad.eq}. 
Furthermore, both $u_-$ and $u_+$ satisfy the limits in \eqref{asympt.eq}.
\end{theorem}

We will prove Theorems~\ref{main.thm} and \ref{main2.thm} by giving the detailed
arguments for the case where (H) holds, and explaining what needs to be modified 
to account for (H').

\begin{remark}\rm
It should be noted that Theorems~\ref{main.thm} and \ref{main2.thm} yield
a complete description of the set of solutions of \eqref{dirad.eq}, and hence
the set of radial solutions of \eqref{dir.eq}. However, for $p\neq2$, there does not
hold a general result about the symmetry of solutions like the famous 
Gidas-Ni-Nirenberg result for the semilinear case \cite{gnn}. It is certainly an interesting
problem to investigate the symmetry of solutions of \eqref{dir.eq} under our 
assumptions but we refrain from going in this direction here.
\end{remark}

As immediate corollaries of Theorems~\ref{main.thm} and \ref{main2.thm}, we have the
following existence results for the problem
\begin{equation}\label{diradf.eq}
\setlength\arraycolsep{0.05cm}
\left\{ \begin{array}{rlll}
-(r^{N-1}\phi_p(u'))' &= & r^{N-1} f(r,u), \quad 0<r<1,\\
u'(0) = u(1) &= &0,
\end{array}\right.
\end{equation}

\begin{corollary}\label{cor1}
Let $p>2$, suppose that (H) holds, and that $1\in(\lam_0,\lam_\infty)$. Then problem
\eqref{diradf.eq} has a unique non-trivial solution $u\in X_p$, such that $u>0$ on 
$[0,1)$.
\end{corollary}

\begin{corollary}\label{cor2}
Let $p>2$, suppose that (H') holds, and that $1\in(\lam_0,\lam_\infty)$. Then problem
\eqref{diradf.eq} has exactly two non-trivial solutions $u_\pm\in X_p$, 
with $\pm u_\pm>0$ on $[0,1)$.
\end{corollary}

\medskip
\noindent{\bf Notation.}
For $h\in C^0([0,1]\x\real)$, we define the Nemitskii mapping $h:Y\to Y$ by
$h(u)(r):=h(r,u(r))$. Then $h:Y\to Y$ is bounded and continuous. We will always use 
the same symbol for a function and for the induced Nemitskii mapping.

When computing estimates, the symbol $C$ will denote positive constants which 
may change from line to line. Their exact values are not essential to the analysis.


\section{The inverse $p$-Laplacian}\label{inverse.sec}

In this section we study an integral operator corresponding to the inverse of
the $p$-Laplacian in the radial setting. Although our main goal in this paper is the
proof of Theorems~\ref{main.thm} and \ref{main2.thm}, which requires $p\ge2$,
the results in this section will be stated in greater generality, for $p>1$.
Let us start by introducing some useful 
notation. For $p>1$, we let
$$
p^*:=\tfrac{1}{p-1} \quad\text{and}\quad p':=\tfrac{p}{p-1}=p^*+1.
$$ 
Then for $\xi,\eta\in\real$, the continuous function $\phi_p:\real\to\real$
satisfies
\begin{equation}\label{phi.eq}
 \phi_p(\xi)=\eta \iff \xi=\phi_{p'}(\eta)=\phi_{p^*+1}(\eta).
\end{equation}

For any $h\in C^0[0,1]$, the problem
\begin{equation}\label{heq.eq}
\setlength\arraycolsep{0.05cm}
\left\{ \begin{array}{rlll}
-(r^{N-1}\phi_p(u'))' &= & r^{N-1} h(r), \quad 0<r<1,\\
u'(0) = u(1) &= &0,
\end{array}\right.
\end{equation}
has a unique solution $u(h) \in X_p$, given by
\begin{equation}\label{formula.eq}
u(h)(r)=\int_r^1 \phi_{p'} \left( \int_0^s \left( \frac{t}{s} \right)^{N-1} h(t) \diff t \right) \diff s.
\end{equation}
The formula \eqref{formula.eq} defines a mapping 
$$S_p:C^0[0,1] \to C^1[0,1], \quad h\mapsto S_p(h)=u(h),$$ that we shall now study.
It will be convenient to rewrite $S_p$ as 
\begin{equation}\label{comp.eq}
S_p=T_p \circ J = \calI \circ \Phi_{p'} \circ J,
\end{equation}
where we define the following operators:
\begin{eqnarray}
\setlength\arraycolsep{0.05cm}
J: C^0[0,1] \to C^1[0,1],  
&J(h)(s) & := \disp\int_0^s \left( \frac{t}{s} \right)^{N-1} h(t) \diff t; \label{J.eq}\notag \\
\Phi_q: C^0[0,1] \to C^0[0,1],  
&\Phi_q(g)(s) & := \phi_q(g(s)), \quad\text{for any} \ q>1; \label{Phi.eq}\notag \\
\calI:C^0[0,1] \to C^1[0,1],  
&\calI(k)(r) & := \int_r^1 k(s) \diff s;\label{defofI.eq}\notag\\
T_p:C^0[0,1] \to C^1[0,1],
& T_p & :=\calI\circ\Phi_{p'} \quad\text{for any} \ p>1. \notag
\end{eqnarray}

It is clear that $\Phi_q, \ \calI$ and $T_p$ are continuous and bounded, 
for any $q,p>1$, and that $\calI$ is linear. Furthermore, $S_p$ is $p^*$-homogeneous.

\begin{remark} \rm
For $1<q<2$, $\Phi_q$ does not map $C^1[0,1]$ into itself, which causes
trouble in differentiating $S_p$ for $p>2$ (i.e. $p'<2$). Nevertheless, if $g\in C^1[0,1]$
has only simple zeros, then $\Phi_q$ maps a neighbourhood of $g$ in $C^1[0,1]$
continuously into $L^1(0,1)$ (see Lemma~2.1 in \cite{br}).
\end{remark}

The following lemma gives important properties of $J$.

\begin{lemma}\label{J.lem}
\item[(i)] $J: C^0[0,1] \to C^1[0,1]$ is well-defined.
\item[(ii)] $J: C^0[0,1] \to C^1[0,1]$ is a bounded linear operator,
with norm $\Vert J\Vert \le 2$.
\item[(iii)] $J$ is Fr\'echet differentiable on $C^0[0,1]$ with $DJ=J$.
\item[(iv)] $J: C^0[0,1] \to C^0[0,1]$ is compact.
\end{lemma}

\begin{proof} (i) Using de l'Hospital's rule, we get 
$$
\lim_{s\to0} J(s)=\tfrac{1}{N-1}\lim_{s\to0}sh(s)=0,
$$ 
and it follows that $J(h)\in C^0[0,1]$ for all $h\in C^0[0,1]$.
Furthermore,
$$
\lim_{s\to0}\frac{J(s)}{s}=\lim_{s\to0}\frac{s^{N-1}h(s)}{Ns^{N-1}}=\frac{h(0)}{N},
$$
and so $J(h)\in C^1[0,1]$, with $J'(0)=h(0)/N$.

(ii) We first have
\begin{equation*}\label{estJ0.eq}
|J(h)|_0 
\le \sup_{0\le s\le 1} \left| \int_0^s \left( \frac{t}{s} \right)^{N-1} \diff t \right| |h|_0
\le  \frac{|h|_0}{N}.
\end{equation*}
Then, since
\begin{align}\label{diffJ.eq}
J(h)'(s) 
&=(1-N)s^{-N}\int_0^s t^{N-1} h(t) \diff t + h(s) \notag \\
&=(1-N)s^{-1}J(h)(s)+h(s) \quad \text{for all} \ s\in(0,1],
\end{align}
we have
\begin{equation*}\label{estJ1.eq}
|J(h)'|_0 
\le  \sup_{0\le s\le 1} \left| (1-N)s^{-N} \int_0^s t^{N-1} \diff t \right| |h|_0 + |h|_0
\le  \frac{2N-1}{N} |h|_0.
\end{equation*}
It follows that $J$ is bounded, with norm
$$\Vert J\Vert:=\sup_{|h|_0=1} |J(h)|_1\le \frac1N + \frac{2N-1}{N}=2.$$

(iii) follows from (ii).

(iv) follows from (ii) and the compact embedding $C^1[0,1] \hookrightarrow C^0[0,1]$.
\end{proof}

We can now state important properties of $S_p$, following from the results above.

\begin{theorem}\label{Spcont.thm}
The mapping $S_p: C^0[0,1] \to C^1[0,1]$ defined by \eqref{comp.eq} is continuous, bounded and compact.
\end{theorem}

The following result is a simple adaptation of Theorem~3.2 of \cite{br} to the present context. This is the first step towards the differentiability of $S_p$.

\begin{proposition}\label{Tp.prop}
\item[(i)] Suppose $1<p<2$. Then $T_p: C^0[0,1] \to B_p$ is $C^1$, and for all
$g,\bar g \in C^0[0,1]$,
\begin{equation}\label{derivofTp.eq}
DT_p(g)\bar g = p^*\calI (|g|^{p^*-1}\bar g).
\end{equation}
\item[(ii)] Suppose $p>2$ and let $g_0\in C^1[0,1]$ have only simple zeros 
(i.e. $g_0(s)=0 \Rightarrow g_0'(s)\neq0$). Then $T_p: C^1[0,1] \to B_p$ is $C^1$
on a neighbourhood $U_0$ of $g_0$ in $C^1[0,1]$, and \eqref{derivofTp.eq} holds
for all $g\in U_0, \ \bar g \in C^1[0,1]$.
\end{proposition}

We are now able to state the main result of this section, about the differentiability of
$S_p$. The statement and the proof of this result are very similar to those of
Theorem~3.4 in \cite{br}.

\begin{theorem}\label{Sp.thm}
\item[(i)] Suppose $1<p<2$. Then $S_p: C^0[0,1] \to B_p$ is $C^1$, and for all
$h,\bar h \in C^0[0,1]$,
\begin{equation}\label{DSp.eq}
DS_p(h)\bar h = p^*\calI (|u(h)'|^{2-p}J(\bar h)),
\end{equation}
where $u(h)=S_p(h)$. Furthermore,
\begin{eqnarray}\label{lineariz.eq}
v=DS_p(h)\bar h \ \Longrightarrow \ v \in B_p \ \text{and} \ 
\setlength\arraycolsep{0.05cm}
\left\{ \begin{array}{rlll}
-(r^{N-1}|u(h)'(r)|^{p-2}v'(r))' &=& p^*r^{N-1}\bar h(r),\\
v'(0) = v(1) &=& 0,
\end{array}\right.
\end{eqnarray}
\item[(ii)] Suppose $p>2$ and let $h_0\in C^0[0,1]$ be such that 
$u(h_0)'(r)=0 \Rightarrow h_0(r)\neq0$. Then there exists a neighbourhood $V_0$
of $h_0$ in $C^0[0,1]$ such that the mapping $h\mapsto |u(h)'|^{2-p}: V_0 \to L^1(0,1)$
is continuous, $S_p:V_0 \to B_p$ is $C^1$, and $DS_p$ satisfies 
\eqref{DSp.eq} and \eqref{lineariz.eq}, for all $h\in V_0, \ \bar h \in C^0[0,1]$. 
\end{theorem}

\begin{proof} (i) In view of \eqref{comp.eq}, 
the differentiability of $S_p: C^0[0,1] \to B_p$ follows from
Lemma~\ref{J.lem}(iii) and Proposition~\ref{Tp.prop}(i). Then, for $h,\bar h \in C^0[0,1]$,
$$
DS_p(h)\bar h = DT_p(J(h))J(\bar h) = p^*\calI(|J(h)|^{p^*-1}J(\bar h)).
$$
Now letting $u(h)=S_p(h)$ and differentiating \eqref{formula.eq} yields
$$
u(h)'=\phi_{p'}(J(h))=\phi_{p^*+1}(J(h)) \ \Longrightarrow \
|u(h)'|^{2-p}=|J(h)|^{p^*-1},
$$
proving \eqref{DSp.eq}, from which the continuity of $DS_p$ follows. We will prove below
that \eqref{lineariz.eq} holds in both cases (i) and (ii).

(ii) The case $p>2$ is more delicate and uses Proposition~\ref{Tp.prop}(ii). We define
$g_0:=J(h_0)\in C^1[0,1]$ and $u_0:=u(h_0)$. Then 
$$
\phi_p(u_0'(r))=g_0(r) \quad\text{and}\quad (r^{N-1}\phi_p(u_0'(r)))'=r^{N-1}h_0(r),
\quad 0\le r\le 1.
$$
We will show that $g_0$ has only simple zeros.
First remark that
$$
g_0(r)=0 \  \Longrightarrow \ \phi_p(u_0'(r))=0 \ \Longrightarrow \ u_0'(r)=0
\  \Longrightarrow \ h_0(r)\neq 0,
$$
by our hypothesis. But now by \eqref{diffJ.eq},
\begin{equation}\label{dg0.eq}
g_0'(r)=(1-N)r^{-1}g_0(r)+h_0(r), \quad 0\le r\le 1.
\end{equation}
Therefore, if $g_0(r)=0$ with $r>0$, then $g_0'(r)=h_0(r)\neq0$. On the other hand,
if $g_0(0)=0$, it follows from \eqref{dg0.eq} that $Ng_0'(0)=h_0(0)\neq0$.
Hence, $g_0$ has only simple zeros. Apart from our statement  \eqref{lineariz.eq} 
which is slightly more precise than its analogue in \cite{br}, 
the proof then follows that of \cite[Theorem~3.4]{br}, using 
Proposition~\ref{Tp.prop}(ii) and the analogue of 
\cite[Lemma~2.1]{br} for the present setting.

To prove statement \eqref{lineariz.eq}, let $v=DS_p(h)\bar h, \ h,\bar h \in C^0[0,1]$.
Then, from \eqref{DSp.eq},
$$
v(r)=p^*\int_r^1|u(h)'(s)|^{2-p}\int_0^s\left( \frac{t}{s} \right)^{N-1}\bar h(t) \diff t \diff s,
\quad r\in[0,1].
$$
Since $|u(h)'|^{2-p} \in L^1(0,1)$ in both cases (i) and (ii), it follows that $v(1)=0$.
Furthermore,
\begin{equation}\label{flat}
v'(r)=-p^*|u(h)'(r)|^{2-p}\int_0^r\left( \frac{t}{r} \right)^{N-1}\bar h(t) \diff t, \quad r\in[0,1] ,
\end{equation}
from which the equation in \eqref{lineariz.eq} easily follows. But \eqref{flat} also implies
$$
|v'(r)|\le C_1 |u(h)'(r)|^{2-p}r, \quad r\in[0,1].
$$
Now since $u(h)'(r)=-\phi_{p'}(\int_0^r(t/r)^{N-1} h(t) \diff t)$, it follows that
$|u(h)'(r)|\le C_2 r^{p'-1}$, so that
$$
|v'(r)|\le C r^{(p'-1)(2-p)+1} = C r^{p^*}, \quad r \in [0,1],
$$
showing that $v'(0)=0$ and finishing the proof.
\end{proof}

\begin{remark}\rm
Note that Theorem~\ref{Sp.thm} reduces to well-known results for $p=2$.
\end{remark}


\section{Properties of solutions}\label{prop.sec}

In this section we discuss some {\it a priori} properties of solutions. 
We first study the sign of solutions and then
we determine their behaviour as $|u|_0\to0/\infty$.

By the results of
Section~\ref{inverse.sec}, $(\lam,u) \in [0,\infty)\x X_p$ is a solution of 
\eqref{dirad.eq} if and only if
\begin{equation}\label{dirop.eq}
F(\lam,u) := u - \lam^{p^*} S_p(f(u)) = 0, \quad (\lam,u)\in [0,\infty)\x Y.
\end{equation}
Note that $F:[0,\infty)\x Y\to Y$ is continuous. Furthermore, $F(0,u)=0 \implies u=0$, so we will only consider solutions in
$$
\calS := \{(\lam,u) \in (0,\infty) \x Y: 
(\lam,u) \ \text{is a solution of} \ \eqref{dirop.eq} \ \text{with} \ u\not\equiv0 \}.
$$

\subsection{The case where (H) holds}

We start with the positivity of solutions.

\begin{proposition}\label{pos.prop}
Let $(\lam,u) \in \calS$. 
Then $u>0$ on $[0,1)$, $u$ is decreasing and satisfies $u'(1)<0$.
\end{proposition}

\begin{proof} Equation \eqref{dirop.eq} yields
$$
u(r)=\lam^{p^*}\int_r^1\phi_{p'}\left(\int_0^s\left( \frac{t}{s} \right)^{N-1} f(t,u(t)) \diff t\right) \diff s.
$$
Since $u\not\equiv0$ is continuous, it follows from (H2) that $u(0)>0$.
Furthermore,
$$
\phi_p(u'(r)) = -\lam \int_0^r\left( \frac{t}{r} \right)^{N-1} f(t,u(t)) \diff t \le 0,
\quad r \in [0,1],
$$
showing that $u'(r) \le 0$ for all $r\in[0,1]$, so $u$ is decreasing on $[0,1]$.
Finally, 
$$
\phi_p(u'(1)) = -\lam \int_0^1 t^{N-1} f(t,u(t)) \diff t < 0.
$$
This implies $u'(1)<0$, from which $u>0$ on $[0,1)$ now follows.
\end{proof}

For the following results, we will use the function $g:[0,1]\x\real\to\real$ defined by
\begin{equation}\label{gdef}
g(r,\xi) := 
\begin{cases} f(r,\xi)/\phi_p(\xi) , \quad \xi\neq0,\\
f_0(r), \quad \xi=0.
\end{cases}
\end{equation}
It follows from our assumptions that $g \in C^0([0,1]\x\real)$ and $g$ satisfies
\begin{equation}\label{boundsforg}
0\le f_\infty(r) \le g(r,\xi) \le f_0(r), \quad (r,\xi) \in [0,1]\x\real.
\end{equation}

\begin{lemma}\label{asympt.lem}
Consider a sequence $\{(\lam_n,u_n)\} \subset \calS$. 
Suppose that $|u_n|_0\to0/\infty$ as $n\to\infty$. Then $\lam_n\to\lam_{0/\infty}$.
\end{lemma}

\begin{proof} 
Setting $v_n:=u_n/|u_n|_0$, we have
\begin{equation}\label{eqforvn}
v_n=S_p(\lam_ng(u_n)\phi_p(v_n))
\end{equation}
or, equivalently,
\begin{equation}\label{slforvn}
\setlength\arraycolsep{0.05cm}
\left\{ \begin{array}{rlll}
-(r^{N-1}\phi_p(v_n'))' &= &\lam r^{N-1} g(u_n)\phi_p(v_n), \quad 0<r<1,\\
v_n'(0) = v_n(1) &= &0,
\end{array}\right.
\end{equation}
where $u\mapsto g(u)$ denotes the Nemitskii mapping induced by $g$. 
Since $v_n>0$ in $[0,1)$ for all $n$,
it follows from \eqref{boundsforg}, \eqref{slforvn}, and the Sturmian-type comparison 
theorem in \cite[Sec. 4]{w} that 
\begin{equation}\label{lambounds.eq}
0<\lam_0 \le \lam_n \le \lam_\infty\le\infty.
\end{equation}

Let us first suppose that hypothesis (H6)(a) holds. 
Then $\lam_\infty<\infty$ and we can suppose that
$\lam_n\to\bar\lam \in [\lam_0,\lam_\infty]$ as $n\to\infty$. Now $|v_n|_0=1$
for all $n$, so $\{\lam_ng(u_n)\phi_p(v_n)\}$ is bounded in $C^0[0,1]$ and 
$\{S_p(\lam_ng(u_n)\phi_p(v_n))\}$ is bounded in $C^1[0,1]$.
Therefore, by \eqref{eqforvn}, we can suppose that $|v_n-\bar v|_0 \to 0$ as $n\to\infty$, 
for some $\bar v \in C^0[0,1]$. It then follows by fairly standard arguments 
(see e.g. the proof of Lemma~5.4 in \cite{d}) that 
$$
g(u_n)\phi_p(v_n) \to  
f_{0/\infty}\phi_p(\bar v) \quad \text{provided} \ |u_n|_0\to0/\infty.
$$
Hence $\bar v$ satisfies $\bar v=S_p(\bar\lam f_{0/\infty}\phi_p(\bar v))$ if 
$|u_n|_0\to0/\infty$, that is,
\begin{equation}\label{slforvbar}
\setlength\arraycolsep{0.05cm}
\left\{ \begin{array}{rlll}
-(r^{N-1}\phi_p(\bar v'))' &= & \bar\lam r^{N-1} f_{0/\infty}\phi_p(\bar v), \quad 0<r<1,\\
\bar v'(0) = \bar v(1) &= &0.
\end{array}\right.
\end{equation}
Now the proof of 
Proposition~\ref{pos.prop} shows that $\bar v>0$ in $[0,1)$ and it follows from the properties
of the eigenvalue problem \eqref{slforvbar} (see \cite[Sec.~5]{w}) that 
$\bar\lam=\lam_{0/\infty}$.

We next suppose that hypothesis (H6)(b) holds, i.e. $N=1$ and
$f_\infty\equiv0$. We first prove that
$\lam_n\to\infty$ if $|u_n|_0\to\infty$. Indeed, if we suppose instead that $\{\lam_n\}$ is 
bounded, then the above argument yields a $\bar v \in C^0[0,1]$ such that 
$v_n \to \bar v$ in $C^0[0,1]$ (up to a subsequence), and it follows that
$g(u_n)\phi_p(v_n)\to0$ in $C^0[0,1]$. Then \eqref{eqforvn} implies $\bar v=0$, 
contradicting $|\bar v|_0=1$.

Regarding the behaviour as $|u_n|_0\to0$, the argument for the case $f_\infty>0$ 
will hold in exactly the same way for $f_\infty\equiv0$ if we can show that $\{\lam_n\}$
is bounded. It follows from \eqref{eqforvn} that
$$
1 =|v_n|_0=v_n(0)= 
\lam_n^{p^*} \int_0^1 \phi_{p'} \left( \int_0^s g(u_n)\phi_p(v_n)\diff t\right) \diff s.
$$
Since $u_n, v_n$ are decreasing on $[0,1]$ and, for any fixed $t\in[0,1]$, 
the mapping $\xi \mapsto f(t,\xi)/\xi^{p-1}$ 
is decreasing on $(0,\infty)$, we have
\begin{align*}
\lam_n^{-p^*} & =\int_0^1 \phi_{p'} \left( \int_0^s g(u_n)\phi_p(v_n)\diff t\right) \diff s \\
					& \ge \int_0^{1/2} \phi_{p'} \left( \int_0^s
							\phi_p(v_n(1/2))\frac{f(t,u_n(0))}{u_n(0)^{p-1}}  \diff t\right) \diff s \\
	& \ge Cv_n(1/2)\min_{0\le t \le\frac12}\left(\frac{f(t,u_n(0))}{u_n(0)^{p-1}}\right)^{p^*}.
\end{align*}
Since $N=1$, it follows from \eqref{dirad.eq} that $u$ is concave for all 
$(\lam,u)\in\calS$. Hence, there is a constant $M>0$ (independent of $n$) such that
$v_n(1/2)\ge M|v_n|_0=M$ for all $n$. 
Furthermore, 
$f(t,u_n(0))/u_n(0)^{p-1}\to f_0(t)>0$ uniformly for $t\in[0,\frac12]$ and so there 
exists $\de>0$ such that $\lam_n^{-p^*}\ge\de$ for $n$ large enough.
Therefore, $\{\lam_n\}$ is bounded.

Note that the arguments above only show that $\lam_{n_k}\to\lam_{0/\infty}$
for a subsequence $\{\lam_{n_k}\}$.
Since they can be applied to any subsequence of $\{\lam_n\}$,
it follows that the whole sequence must converge. This concludes the proof.
\end{proof}

\begin{remark}\label{lambounds.rem}\rm
The proof of \eqref{lambounds.eq} shows that 
$\lam_0\le\lam\le\lam_\infty$ for all $(\lam,u)\in\calS$.
\end{remark}

\subsection{The case where (H') holds}

In this case we consider solutions in the sets
$$
\calS^\pm := \{(\lam,u) \in (0,\infty) \x Y: 
(\lam,u) \ \text{is a solution of} \ \eqref{dirop.eq} \ \text{with} \ 
u\not\equiv0, \ \text{and} \  \pm u \ge0 \}.
$$
The following result can be proved as Proposition~\ref{pos.prop}, using (H2') instead
of (H2).
\begin{proposition}\label{pos'.prop}
Let $(\lam,u) \in \calS^\pm$. 
Then $\pm u>0$ on $[0,1)$, $\pm u$ is decreasing and satisfies $\pm u'(1)<0$.
\end{proposition}

Regarding the asymptotic behaviour, we have
\begin{lemma}\label{asympt'.lem}
Consider a sequence $\{(\lam_n,u_n)\} \subset \calS^\pm$. 
Suppose that $|u_n|_0\to0/\infty$ as $n\to\infty$. Then $\lam_n\to\lam_{0/\infty}$.
\end{lemma}

\begin{proof}
The proof is the same as for Lemma~\ref{asympt.lem} in the case where
$\{(\lam_n,u_n)\} \subset \calS^+$. In case $u_n\le0$, as similar proof can be carried
out, setting $v_n:=-u_n/|u_n|_0\ge0$ and remarking that, with this new definition, 
$v_n$ still satisfies \eqref{eqforvn}.
\end{proof}

\begin{remark}\label{lambounds'.rem}\rm
We also have 
$\lam_0\le\lam\le\lam_\infty$ for all $(\lam,u)\in\calS^\pm$.
\end{remark}


\section{Local bifurcation}\label{cranrab.sec}

To prove Theorems~\ref{main.thm} and \ref{main2.thm}, we begin with a local bifurcation
result in the spirit of Crandall and Rabinowitz \cite{cr}. This will allow us to start off
the bifurcating branch from the line of trivial solutions at the point $(\lam_0,0)$
in $\real\x Y$. Crandall and Rabinowitz' original result pertained
to semilinear equations, i.e. $p=2$. A first generalization to $p>2$ was given
in \cite{gs} for a problem very similar to \eqref{dirad.eq}. The main difference in our
setting is that we allow the asymptote $f_0$ to depend on $r$, so that we get the
weighted eigenvalue problem ($\mathrm{E}_0$) instead of problem \eqref{hom}.

In the following, we assume that the principal eigenfunction $v_0$ given in 
Lemma~\ref{eigen.lem} is normalized so that
$$
\int_0^1 r^{N-1} f_0(r) |v_0(r)|^p \diff r = 1.
$$
We define the subspace
$$
Z:=\{z \in Y: \disp \int_0^1 r^{N-1} f_0 |v_0|^{p-2}v_0 z \diff r = 0\}
$$
and we remark that
\begin{equation}\label{decomp}
Y=\mathrm{span}\{v_0\}\oplus W.
\end{equation}
To be able to discuss later cases (H) and (H'),
it will be convenient to state our local bifurcation result more generally, 
in terms of the function 
$G: \real^2 \x Z \to Y$ defined by
$$
G(s,\lam,z):=
\begin{cases}
v_0+z - S_p(\lam f(sv_0+sz)/\phi_p(s)), \quad s\neq0,\\
v_0+z - S_p(\lam f_0\phi_p(v_0+z)), \quad s=0.
\end{cases}
$$
Note that $G(s,\lam,z)=F(\lam,s(v_0+z))/s$ for all $s\neq0$, where 
$F:\real\x Y \to Y$ was defined in \eqref{dirop.eq}. Also, it follows from
the definitions of $\lam_0$ and $v_0$ that $G(0,\lam_0,0)=0$.

\begin{lemma}\label{cranrab.lem}
Let $p>2$ and
suppose that (H1) and (H4) hold. There exist $\eps>0$, a neighbourhood
$U$ of $(\lam_0,0)$ in $\real\x Z$  and a continuous mapping 
$s \mapsto (\lam(s),z(s)): (-\eps,\eps) \to U$ such that 
$(\lam(0),z(0))=(\lam_0,0)$ and
$$
\{(s,\lam,z)\in(-\eps,\eps) \x U: G(s,\lam,z)=0\}
=\{(s,\lam(s),z(s)):s\in(-\eps,\eps)\}.
$$
\end{lemma}

\begin{proof}
Our proof follows that of \cite[Theorem~1]{gs} but we give it here for
completeness. Under hypotheses (H1) and (H4), it is easily seen that $G$ is continuous.
It follows from Theorem~\ref{Sp.thm}(ii) that $S_p$ is $C^1$ in a neighbourhood
of $\lam_0 f_0 \phi_p(v_0)$ in $Y$. A routine verification then shows that there
is a neighbourhood $U$ of $(0,\lam_0,0)$ in $\real^2\x Z$ such that the mapping
$(\lam,z) \mapsto G(s,\lam,z)$ is differentiable in 
$$
A_s:=\{(\lam,z) \in \real\x Z: (s,\lam,z) \in U\},
$$
for any $s\in\real$ such that $A_s \neq\emptyset$. Furthermore, the Fr\'echet derivative
$D_{(\lam,z)} G$ is continuous on $U$. Since 
\begin{equation}\label{opeigen}
v_0=S_p(\lam_0 f_0 \phi_p(v_0)), 
\end{equation}
it follows
from \eqref{DSp.eq} that
\begin{align}\label{derivG.eq}
& D_{(\lam,z)}G(0,\lam_0,0)(\bar\lam,\bar z)= \\
& \bar z - \lam_0(p^*)^{-1} DS_p(\lam f_0 \phi_p(v_0))f_0|v_0|^{p-2} \bar z
-p^* (\bar\lam/\lam_0)v_0, \quad  (\bar\lam, \bar z) \in \real \x Z.
\end{align}
To conclude the proof using the implicit function theorem as stated in Appendix~A
of \cite{cr}, we need only check that 
$D_{(\lam,z)}G(0,\lam_0,0): \real \x Z \to Y$ is an isomorphism. 

Let us first show that the mapping
$$
L\bar z :=  \lam_0(p^*)^{-1} DS_p(\lam f_0 \phi_p(v_0))f_0|v_0|^{p-2} \bar z
$$
leaves the subspace $Z$ invariant. Suppose $\bar z \in Z$ and let $z=L\bar z$.
By \eqref{lineariz.eq} and \eqref{opeigen}, we have
\begin{equation}\label{eqforz}
\setlength\arraycolsep{0.05cm}
\left\{ \begin{array}{rlll}
-(r^{N-1}|v_0'|^{p-2}z')' &= &\lam_0 r^{N-1} f_0|v_0|^{p-2}\bar z, \quad 0<r<1,\\
z'(0) = z(1) &= &0.
\end{array}\right.
\end{equation}
Multiplying both sides of the equation by $v_0$ and integrating by parts twice yields
$$
\int_0^1 r^{N-1} f_0 |v_0|^{p-2}v_0 z \diff r = 
\int_0^1 r^{N-1} f_0 |v_0|^{p-2}v_0 \bar z \diff r = 0,
$$
showing that $z \in Z$. In view of the decomposition \eqref{decomp},
$$
D_{(\lam,z)}G(0,\lam_0,0)(\bar\lam,\bar z)=0 \implies 
\bar\lam=0 \ \text{and} \ \bar z = L \bar z.
$$
Then $\bar z$ is a solution of \eqref{eqforz} and an argument similar to the proof
of \cite[Theorem~7]{gs} shows that there exists $c\in\real$ such that $\bar z= c v_0$.
Since $\bar z \in Z$, it follows that $c=0$, showing that the null space
$N(D_{(\lam,z)}G(0,\lam_0,0))=\{0\}$.

Finally, from \eqref{decomp} and the invariance of $Z$ under $L$,
$D_{(\lam,z)}G(0,\lam_0,0)$ is isomorphically equivalent to the operator 
$T: \real\x Z \to \real\x Z$ defined by
$$
T(\bar\lam,\bar z)=(\bar\lam,\bar z) -((1+p^*/\lam_0)\bar\lam, L\bar z).
$$
It follows from Theorem~\ref{Spcont.thm} that 
$T$ is a compact perturbation of the identity on
$\real\x Z$. Therefore, the triviality of $N(D_{(\lam,z)}G(0,\lam_0,0))$ implies that 
$D_{(\lam,z)}G(0,\lam_0,0)$ is an isomorphism, finishing the proof.
\end{proof}

\begin{remark}\rm
As noted earlier, the above result was presented in \cite{gs} in the case where 
$f_0\equiv1$.
However, the proof relies heavily on the differentiability properties of the integral operator
$S_p$, given by Theorem~\ref{Sp.thm}, and the arguments establishing
these properties in \cite{gs} seem incomplete. Hence, in addition to the slightly more
general context dealt with here, the present work completes the proof of 
\cite[Theorem~1]{gs}.
\end{remark}

\begin{remark}\label{diff.rem}\rm
Since the differentiability results in Theorem~\ref{Sp.thm} cover the whole range $p>1$,
we first had some hope to obtain bifurcation for all $p>1$. It turns out that the integration
by parts arguments involved in the proof of Lemma~\ref{cranrab.lem} require at least
$p\ge 1+1/N$ (for the boundary terms to vanish). Unfortunately, the differentiability of
the function $G$ in the present functional setting requires $p\ge2$, and we have not
been able to find another suitable setting allowing for $p<2$.
\end{remark}

We can now state the local bifurcation results for equation \eqref{dirop.eq}.

\begin{theorem}\label{loc1.thm}
Let $p>2$ and
suppose that (H) holds. There exist $\eps_0\in(0,\eps)$ and a neighbourhood $U_0$ of
$(\lam_0,0)$ in $\real\x Y$ such that
\begin{equation}\label{loc1.eq}
\{(\lam,u) \in U_0: F(\lam,u)=0\}=\{(\lam(s),s(v_0+z(s))): s \in [0,\eps_0)\}.
\end{equation}
\end{theorem}

\begin{proof}
It follows from Lemma~\ref{cranrab.lem} that $(\lam(s),s(v_0+z(s)))$ is a solution of 
\eqref{dirop.eq} for all $s \in [0,\eps)$. 
To prove the reverse inclusion in \eqref{loc1.eq},
let us first remark that, by Proposition~\ref{pos.prop}, 
$s\in(-\eps,0)$ yields no solutions of \eqref{dirop.eq}. Furthermore, a compactness
argument similar to that in \cite[p.39]{gs} shows that any solution in
a small enough neighbourhood of $(\lam_0,0)$ in $\real\x Y$ must have the form
$(\lam(s),s(v_0+z(s)))$ for some $s \in [0,\eps)$. This completes the proof.
\end{proof}

\begin{theorem}\label{loc2.thm}
Let $p>2$ and
suppose that (H') holds. There exist $\eps_1\in(0,\eps)$ and a neighbourhood 
$U_1$ of $(\lam_0,0)$ in $\real\x Y$ such that
\begin{equation}\label{loc2.eq}
\{(\lam,u) \in U_1: F(\lam,u)=0\}=\{(\lam(s),s(v_0+z(s))): s \in (-\eps_1,\eps_1)\},
\end{equation}
with
\begin{equation}\label{loc2-.eq}
\{(\lam(s),s(v_0+z(s))): s \in (-\eps_1,0)\}\subset\calS^-
\end{equation}
and
\begin{equation}\label{loc2+.eq}
\{(\lam(s),s(v_0+z(s))): s \in (0,\eps_1)\}\subset\calS^+.
\end{equation}
\end{theorem}

\begin{proof}
The local characterization of solutions in \eqref{loc2.eq} follows similarly to 
\eqref{loc1.eq} in Theorem~\ref{loc1.thm}. For $\eps_1>0$ small enough,
statements \eqref{loc2-.eq} and
\eqref{loc2+.eq} follow from the construction of the solutions $(\lam(s),s(v_0+z(s)))$.
\end{proof}


\section{Global continuation}\label{global.sec}

Our goal in this final section is to complete the proofs of Theorems~\ref{main.thm} and
\ref{main2.thm}. Namely, we will first show that the local curves of solutions obtained
in Section~\ref{cranrab.sec} can be parametrized by $\lam$ and then we will prove that
they can be extended globally.

\subsection{Proof of Theorem~\ref{main.thm}}

We begin with a non-degeneracy result implying that, in fact, through any non-trivial
solution of \eqref{dirop.eq}, there passes a (local) continuous curve of solutions,
parametrized by $\lam$.

\begin{lemma}\label{IFT.lem}
The function $F \in C([0,\infty)\x Y, Y)$ defined in \eqref{dirop.eq} is continuously
differentiable in a neighbourhood of any point $(\lam,u) \in \calS$, with
$$
D_u F(\lam,u) v = v - \lam DS_p(\lam f(u)) \partial_2 f(u) v, \quad v\in Y.
$$
Furthermore, for any $(\lam,u) \in \calS$, $D_u F(\lam,u):Y\to Y$ is an isomorphism.
\end{lemma}

\begin{proof}
The statement about the differentiability of $F$ follows from Theorem~\ref{Sp.thm}(ii)
and Proposition~\ref{pos.prop}. Furthermore, we see that
$D_u F(\lam,u):Y\to Y$ is a compact perturbation of the identity. Therefore,
to show that it is an isomorphism, we only need to prove that
$N(D_u F(\lam,u))=\{0\}$.
Let $v\in N(D_u F(\lam,u))$. By \eqref{lineariz.eq}, we have 
\begin{equation}\label{eqforv}
\setlength\arraycolsep{0.05cm}
\left\{ \begin{array}{rlll}
-(r^{N-1}|u'|^{p-2}v')' &= & p^* \lam r^{N-1} \partial_2 f(u) v, \quad 0<r<1,\\
v'(0) = v(1) &= &0.
\end{array}\right.
\end{equation}
Multiplying the equation in \eqref{eqforv} by $u$, that in \eqref{dirad.eq} by $v$, 
subtracting and integrating by parts yield
\begin{equation}\label{lagrange}
r^{N-1}|u'|^{p-2}(uv'-u'v)(r)=\lam\int_0^r s^{N-1}[f(u)-p^*\partial_2f(u)u]v\diff s,
\quad r\in[0,1].
\end{equation}
Suppose that $v\not\equiv0$, and let $r_1>0$ be the smallest
positive zero of $v$. Without loss of generality, we can suppose $v>0$ on $(0,r_1)$. 
If $r_1<1$, we have $u(r_1)v'(r_1)<0$. However by (H3),
\begin{equation*}\label{signofv}
r_1^{N-1}|u'(r_1)|^{p-2}u(r_1)v'(r_1)=
\lam\int_0^{r_1} s^{N-1}[f(u)-p^*\partial_2f(u)u]v\diff s \ge 0,
\end{equation*}
a contradiction. If $r_1=1$, (H3) and Proposition~\ref{pos.prop} imply
$$
0=|u'(1)|^{p-2}u(1)v'(1)=
\lam \int_0^{1} s^{N-1}[f(u)-p^*\partial_2f(u)u]v\diff s > 0,
$$
again a contradiction. Hence, $v\not\equiv0$ is impossible and so 
$N(D_u F(\lam,u))=\{0\}$.
\end{proof}

By Remark~\ref{lambounds.rem}, Theorem~\ref{loc1.thm} and Lemma~\ref{IFT.lem}, 
the implicit function theorem yields a maximal open interval $(\lam_0,\wt\lam)$ with
$\lam_0<\wt\lam\le\lam_\infty$ and a mapping
$u \in C^1((\lam_0,\wt\lam),Y)$ such that $(\lam,u(\lam))\in\calS$ for all 
$\lam\in(\lam_0,\wt\lam)$, and $\lim_{\lam\to\lam_0}u(\lam)=0$. 
Let us show that $\wt\lam=\lam_\infty$. Suppose by contradiction that 
$\lam_0<\wt\lam<\lam_\infty\le\infty$, and consider a sequence $\lam_n\to\wt\lam$.
If $|u(\lam_n)|_0$ is unbounded, it follows by Lemma~\ref{asympt.lem} that
$\wt\lam=\lam_\infty$ and we are done. On the other hand, if $|u(\lam_n)|_0$ is 
bounded, a compactness argument similar to that yielding the convergence of 
$\{v_n\}$ in the proof of Lemma~\ref{asympt.lem} shows that there exists
$\wt u\in Y$ such that $u(\lam_n)\to\wt u$ (up to a subsequence), and
\begin{equation*}
\wt u = S_p(\wt\lam f(\wt u)).
\end{equation*}
Note that, by Lemma~\ref{asympt.lem}, we cannot have $u\equiv0$.
Hence, $(\wt\lam,\wt u)\in\calS, \ F(\wt\lam,\wt u)=0$, and so by Lemma~\ref{IFT.lem}
and the implicit function theorem, we can extend the curve $u(\lam)$ through the point
$(\wt\lam,\wt u)$, contradicting the maximality of $\wt\lam$. Therefore, 
$\wt\lam=\lam_\infty$, and we have a solution curve
$$
\calS_0:=\{(\lam,u(\lam)):\lam\in(\lam_0,\lam_\infty)\}\subset\calS.
$$

We next prove that $\lim_{\lam\to\lam_\infty}|u(\lam)|_0=\infty$. In the case where
(H6)(a) holds, this readily follows by the above argument for if $|u(\lam)|_0$ were 
bounded as $\lam\to\lam_\infty<\infty$, we could continue the solution curve beyond
$\lam=\lam_\infty$. In case (H6)(b) holds, the result follows from

\begin{lemma}\label{blowup.lem}
Suppose that (H6)(b) holds, and consider 
$(\lam_n,u_n)\in\calS$ with $\lam_n\to\infty$. Then $|u_n|_0\to\infty$.
\end{lemma}

\begin{proof}
By contradiction, suppose there exists a constant $R>0$ and a subsequence, still denoted by $(\lam_n,u_n)$, such that $|u_n|_0\le R$ for all $n$. Then, by (H2) and (H3),
\begin{align*}
|u_n|_0=u_n(0)	&= \int_0^1\phi_{p'}\left(\int_0^s\lam_n f(t,u_n)\diff t \right)\diff s\\
						&\ge \lam_n^{p^*} \int_0^{1/2} \phi_{p'}
						\left(\int_0^s g(t,u_n)u_n(t)^{p-1}\diff t \right)\diff s\\
						&\ge \lam_n^{p^*} u_n(1/2) \int_0^{1/2} \phi_{p'}
						\left(\int_0^s g(t,R)\diff t \right)\diff s\\
						&\ge \lam_n^{p^*} C|u_n|_0,
\end{align*}
where the last inequality follows from the concavity of the solutions $u_n$ on $[0,1]$.
Hence $\lam_n^{p^*}\le C^{-1}<\infty$, a contradiction.
\end{proof}

We still need to prove the uniqueness statement of Theorem~\ref{main.thm}, that is,
$\calS_0=\calS$. Suppose instead that there exists 
$(\bar\lam,\bar u) \in \calS\sm\calS_0$, and let $\calS_1$ be the connected subset of
$\calS$ such that $(\bar\lam,\bar u) \in \calS_1$. It follows by Lemma~\ref{IFT.lem}
that $\calS_1$ is a smooth curve, parametrized by $\lam$ in a maximal interval 
$I_1$. In fact, the previous arguments imply that $I_1=(\lam_0,\lam_\infty)$. Let us 
denote by $u_1:(\lam_0,\lam_\infty)\to Y$ the parametrization of $\calS_1$ and
consider a sequence $\lam_n\to\lam_0$. Since $|u_1(\lam_n)|_0$ is bounded
by Lemma~\ref{asympt.lem}, it follows that there exists $u_0 \in Y$ such that
$u(\lam_n)\to u_0$ in $Y$ as $n\to\infty$. Then by continuity, we have
\begin{equation*}
u_0=S_p(\lam_0f(u_0)).
\end{equation*}
Since $u_1(\lam_n)\ge0$ for all $n$, it follows that $u_0\ge0$. We will show that, in fact,
$u_0\equiv0$. Hence we will have $(\lam_n,u_1(\lam_n))\to (\lam_0,0)$ in $Y$
and, by the characterization \eqref{loc1.eq} in Theorem~\ref{loc1.thm}, 
$\calS_1=\calS_0$.
If $u_0\not\equiv0$, we set $w_0=u_0/|u_0|_0$. Then $w_0\ge0$ and satisfies
\begin{equation}\label{final.eq}
w_0=S_p(\lam_0g(u_0)\phi_p(w_0)).
\end{equation}
Having in mind (H3) and \eqref{boundsforg}, it follows from the comparison
theorem of \cite[Sec.~4]{w} applied to \eqref{final.eq} and ($\mathrm{E}_0$) that
we must have $w_0\equiv0$. This contradiction finishes the proof of 
Theorem~\ref{main.thm}.\hfill $\Box$

\subsection{Proof of Theorem~\ref{main2.thm}}

We start with the analogue of Lemma~\ref{IFT.lem} under hypothesis (H').

\begin{lemma}\label{IFT'.lem}
The function $F \in C([0,\infty)\x Y, Y)$ defined in \eqref{dirop.eq} is continuously
differentiable in a neighbourhood of any point $(\lam,u) \in \calS^\pm$, with
$$
D_u F(\lam,u) v = v - \lam DS_p(\lam f(u)) \partial_2 f(u) v, \quad v\in Y.
$$
Furthermore, for any $(\lam,u) \in \calS^\pm$, $D_u F(\lam,u):Y\to Y$ is an isomorphism.
\end{lemma}

\begin{proof}
The proof is almost identical to that of Lemma~\ref{IFT.lem}, so we only indicate
the minor modifications. The differentiability part follows as in Lemma~\ref{IFT.lem}, 
using Theorem~\ref{Sp.thm}(ii), and Proposition~\ref{pos'.prop} instead of
Proposition~\ref{pos.prop}. The non-singularity of $D_u F(\lam,u):Y\to Y$ follows in the
same way if $(\lam,u)\in\calS^+$. For $(\lam,u)\in\calS^-$, we proceed in a similar 
manner, considering $v\in N(D_u F(\lam,u))$. Then the identity \eqref{lagrange}
still holds, and we suppose by contradiction that $v>0$ on a maximal interval
$(0,r_2)$, with $v(r_2)=0$ and $v'(r_2)<0$. If $r_2<1$, we have $u(r_2)v'(r_2)>0$ 
while (H3') implies
\begin{equation*}
r_2^{N-1}|u'(r_2)|^{p-2}u(r_2)v'(r_2)=
\lam\int_0^{r_2} s^{N-1}[f(u)-p^*\partial_2f(u)u]v\diff s \le 0,
\end{equation*}
a contradiction. On the other hand, if $r_2=1$, it follows from (H3') and 
Proposition~\ref{pos'.prop} that
$$
0=|u'(1)|^{p-2}u(1)v'(1)=
\lam \int_0^{1} s^{N-1}[f(u)-p^*\partial_2f(u)u]v\diff s < 0,
$$
another contradiction. Hence $v\equiv0$ and
$N(D_u F(\lam,u))=\{0\}$.
\end{proof}

Now, similarly to the proof of Theorem~\ref{main.thm}, using 
Remark~\ref{lambounds'.rem}, Theorem~\ref{loc2.thm} and Lemma~\ref{IFT'.lem}, 
the implicit function theorem yields maximal open intervals $(\lam_0,\wt\lam_\pm)$
with $\lam_0<\wt\lam_\pm\le\lam_\infty$ and two solution curves
$u_\pm \in C^1((\lam_0,\wt\lam_\pm),Y)$. 
It follows as in the proof of Theorem~\ref{main.thm} that 
$\wt\lam_\pm=\lam_\infty$, so we get two global solution curves
$$
\calS_0^\pm=\{(\lam,u_\pm(\lam)):\lam\in(\lam_0,\lam_\infty)\}\subset\calS.
$$

Furthermore, $\lim_{\lam\to\lam_\infty}|u(\lam)|_0=\infty$ follows as in the proof 
of Theorem~\ref{main.thm}, using the version of Lemma~\ref{blowup.lem} holding under 
hypothesis (H'):

\begin{lemma}\label{blowu'p.lem}
Suppose that (H6)(b) holds, and consider 
$(\lam_n,u_n)\in\calS^\pm$ with $\lam_n\to\infty$. Then $|u_n|_0\to\infty$.
\end{lemma}

\begin{proof}
The proof is the same as that of Lemma~\ref{blowup.lem} when 
$(\lam_n,u_n)\in\calS^+$. 

For $(\lam_n,u_n)\in\calS^-$, if $|u_n|_0\le R$, it follows by (H2') and (H3') that
\begin{align*}
|u_n|_0=-u_n(0)&=-\int_0^1\phi_{p'}\left(\int_0^s\lam_n f(t,u_n)\diff t \right)\diff s\\
						&\ge \lam_n^{p^*} \int_0^{1/2} \phi_{p'}
						\left(\int_0^s g(t,u_n)|u_n(t)|^{p-1}\diff t \right)\diff s\\
						&\ge \lam_n^{p^*} |u_n(1/2)| \int_0^{1/2} \phi_{p'}
						\left(\int_0^s g(t,R)\diff t \right)\diff s\\
						&\ge \lam_n^{p^*} C|u_n|_0,
\end{align*}
showing that the sequence $\{\lam_n\}$ must be bounded, a contradiction.
\end{proof}

Using the characterization \eqref{loc2.eq} in Theorem~\ref{loc2.thm}, 
it follows similarly to the last part of
the proof of Theorem~\ref{main.thm} that $\calS=\calS_0^-\cup\calS_0^+$. 
Hence, to finish the proof of Theorem~\ref{main2.thm}, we only need to prove that
$\calS_0^-\subset\calS^-$ and $\calS_0^+\subset\calS^+$.

Let $\calC_0^\pm:=\calS_0^\pm\cap\calS^\pm$. 
First, we know from Theorem~\ref{loc2.thm}
that, for $\lam$ close to $\lam_0$, $(\lam,u_\pm(\lam))\in\calS^\pm$, so that 
$\calC_0^\pm\neq\emptyset$.
Thus, we need only show that $\calC_0^\pm$
is both open and closed in $\calS_0^\pm$, for the product topology inherited from 
$\real\x Y$. We will only consider $\calC_0^+$, the proof for $\calC_0^-$ is similar.

For $(\lam,u)\in \calC_0^+$, it follows from 
Proposition~\ref{pos'.prop} that $u>0$ on $[0,1)$ (with $u(1)=0$). 
If $(\mu,v)\in \calS_0^+$ with 
$|\mu-\lam|+|v-u|_0$ small enough, we will have $\lam\in(\lam_0,\lam_\infty)$ and 
$0\not\equiv v\ge0$ on $[0,1]$, hence $(\mu,v)\in\calC_0^+$. 
This proves that $\calC_0^+$ is open in $\calS_0^+$.

Now consider a sequence $\{(\lam_n,u_n)\}\subset\calC_0^+$ and suppose there exists 
$(\lam,u)\in \calS_0^+$ such that $(\lam_n,u_n) \to (\lam,u)$.
By continuity, $F(\lam,u)=0$, and we have $u\ge0$. Since $\lam>\lam_0$, 
it follows from Lemma~\ref{asympt'.lem} that $u\not\equiv0$. Hence,
$(\lam,u)\in\calS_0^+$ and $\calC_0^+$ is closed in $\calS_0^+$.
Thus, $\calC_0^+=\calS_0^+$, and it follows in a similar way that $\calC_0^-=\calS_0^-$,
showing that $\calS_0^\pm\subset\calS^\pm$, and so actually 
$\calS_0^\pm=\calS^\pm$. This completes the proof of Theorem~\ref{main2.thm}.
\hfill$\Box$

\end{document}